%% file: 0-KFJLT-guarantees-main-arxiv-v2.tex
\documentclass[review]{elsarticle}

\input{packages-macros.tex}

\journal{Journal of \LaTeX\ Templates}

\begin{document}

\begin{frontmatter}

\title{Guarantees for the Kronecker Fast Johnson--Lindenstrauss Transform Using a Coherence and Sampling Argument}

\author{Osman~Asif~Malik\corref{cor1}}
\ead{osman.malik@colorado.edu}
\author{Stephen~Becker}

\address{Department of Applied Mathematics, University of Colorado Boulder, USA}


\cortext[cor1]{Corresponding author}


\begin{abstract}
	In the recent paper [Jin, Kolda \& Ward, arXiv:1909.04801], it is proved that the Kronecker fast Johnson--Lindenstrauss transform (KFJLT) is, in fact, a Johnson--Lindenstrauss transform, which had previously only been conjectured. In this paper, we provide an alternative proof of this, for when the KFJLT is applied to Kronecker vectors, using a coherence and sampling argument. Our proof yields a different bound on the embedding dimension, which can be combined with the bound in the paper by Jin et al.\ to get a better bound overall. As a stepping stone to proving our result, we also show that the KFJLT is a subspace embedding for matrices with columns that have Kronecker product structure. Lastly, we compare the KFJLT to four other sketch techniques in numerical experiments on both synthetic and real-world data.
\end{abstract}


\end{frontmatter}


\section{Introduction} \label{sec:intro}

\input{1-KFJLT-guarantees-intro-arxiv-v2.tex}

\section{Other Related Work} \label{sec:related-work}

\input{2-KFJLT-guarantees-related-work-arxiv-v2.tex}

\section{Preliminaries} \label{sec:preliminaries}

\input{3-KFJLT-guarantees-prerequisites-arxiv-v2.tex}

\section{Main Results} \label{sec:main-results}

\input{4-KFJLT-guarantees-main-result-arxiv-v2.tex}

\section{Numerical Experiments} \label{eq:experiments}

\input{5-KFJLT-guarantees-experiments-arxiv-v2.tex}

\section{Conclusion} \label{sec:conclusion}

\input{6-KFJLT-guarantees-conclusions-arxiv-v2.tex}

\section*{Acknowledgments}

We would like to thank Tammy Kolda, Ruhui Jin and Rachel Ward for providing feedback on an early version of this work. We also thank the anonymous reviewers for their comments which helped improve the paper.

This material is based upon work supported by the National Science Foundation under Grant No.\ ECCS-1810314.

\bibliographystyle{plainnat}
\bibliography{library}

\end{document}

%% file: packages-macros.tex
\usepackage{amsthm}
\usepackage{amssymb}
\usepackage{hyperref}
\usepackage{booktabs}
\usepackage{mathtools} \mathtoolsset{showonlyrefs=true} \MakeRobust{\eqref}
\usepackage{bm}
\usepackage[mathscr]{eucal}
\usepackage[numbers]{natbib}
\usepackage[utf8]{inputenc}

\theoremstyle{plain}
\newtheorem{theorem}{Theorem} \numberwithin{theorem}{section}

\newtheorem{lemma}[theorem]{Lemma}

\theoremstyle{definition}
\newtheorem{remark}[theorem]{Remark}

\newtheorem{definition}[theorem]{Definition}

\hypersetup{
	colorlinks=true,
	citecolor=MidnightBlue,
	urlcolor=Bittersweet,
	linkcolor=MidnightBlue,
}

\newcommand{\F}{\textup{F}}

\newcommand\Pb{\mathbb{P}}

\newcommand\Rb{\mathbb{R}}

\newcommand{\Abf}[0]{\mathbf{A}}
\newcommand{\Bbf}[0]{\mathbf{B}}

\newcommand{\Dbf}[0]{\mathbf{D}}

\newcommand{\Hbf}[0]{\mathbf{H}}

\newcommand{\Mbf}[0]{\mathbf{M}}

\newcommand{\Rbf}[0]{\mathbf{R}}
\newcommand{\Sbf}[0]{\mathbf{S}}

\newcommand{\Ubf}[0]{\mathbf{U}}
\newcommand{\Vbf}[0]{\mathbf{V}}

\newcommand{\Xbf}[0]{\mathbf{X}}
\newcommand{\Ybf}[0]{\mathbf{Y}}

\newcommand{\abf}[0]{\mathbf{a}}

\newcommand{\ebf}[0]{\mathbf{e}}

\newcommand{\pbf}[0]{\mathbf{p}}
\newcommand{\qbf}[0]{\mathbf{q}}

\newcommand{\ubf}[0]{\mathbf{u}}
\newcommand{\vbf}[0]{\mathbf{v}}
\newcommand{\wbf}[0]{\mathbf{w}}
\newcommand{\xbf}[0]{\mathbf{x}}
\newcommand{\ybf}[0]{\mathbf{y}}
\newcommand{\zbf}[0]{\mathbf{z}}

\newcommand{\Omegabf}[0]{\bm{\Omega}}
\newcommand{\Phibf}[0]{\bm{\Phi}}

\newcommand{\Sigmabf}[0]{\bm{\Sigma}}

\newcommand{\Fc}[0]{\mathcal{F}}

\newcommand{\Xc}{\mathcal{X}}
\newcommand{\Yc}{\mathcal{Y}}

\newcommand{\Xe}{\boldsymbol{\mathscr{X}}}
\newcommand{\Ye}{\boldsymbol{\mathscr{Y}}}

\newcommand{\rank}[0]{\operatorname{rank}}

\newcommand{\range}[0]{\operatorname{range}}

\newcommand{\col}[0]{\operatorname{col}}

\DeclareMathOperator*{\argmin}{arg\,min}

\DeclareMathOperator*{\kron}{\bigotimes}
\DeclareMathOperator*{\kr}{\bigodot}

\newcommand{\JLT}{\operatorname{JLT}}

\newcommand{\defeq}{\stackrel{\text{\tiny \textnormal{def}}}{=}}  

\newcommand{\OPT}{\textup{OPT}}

%% file: 1-KFJLT-guarantees-intro-arxiv-v2.tex
The Johnson--Lindenstrauss lemma, which was introduced by \citet{johnson1984}, is the following fact. 
\begin{theorem}[Johnson--Lindenstrauss lemma \citep{dasgupta2003}] \label{thm:JL-lemma}
	Let $\varepsilon \in (0, 1)$ be a real number, let $\Xc \subseteq \Rb^{I}$ be a set of $N$ points, and suppose $J \geq C \varepsilon^{-2} \log N$, where $C$ is an absolute constant. Then there exists a map $f : \Rb^I \rightarrow \Rb^J$ such that for all $\xbf, \ybf \in \Xc$,
	\begin{equation}
		(1-\varepsilon) \| \xbf - \ybf \|_2^2 \leq \| f(\xbf) - f(\ybf) \|_2^2 \leq (1+\varepsilon) \| \xbf - \ybf \|_2^2.
	\end{equation}
\end{theorem}
Any mapping $f$ which has this property is called a Johnson--Lindenstrauss transform. Typically, such transforms are random maps, which motivates the following, more precise, definition.
\begin{definition}[Johnson--Lindenstrauss transform \citep{woodruff2014}] \label{def:JLT}
	A probability distribution on a family of maps 
	$\Fc$, where each $f \in \Fc$ maps $\Yc \subseteq \Rb^I$ to $\Rb^J$, is a \emph{Johnson--Lindenstrauss transform} with parameters $\varepsilon$, $\delta$, and $N$, or $\JLT(\varepsilon, \delta, N)$, on $\Yc$ if, for any subset $\Xc \subseteq \Yc$ containing $N$ elements, the probability of drawing a map $f \in \Fc$ which satisfies  
	\begin{equation}
		(\forall \xbf, \ybf \in \Xc) \;\;\;\; (1-\varepsilon) \| \xbf - \ybf \|_2^2 \leq \| f(\xbf) - f(\ybf) \|_2^2 \leq (1+\varepsilon) \| \xbf - \ybf \|_2^2
	\end{equation}
	is at least $1-\delta$.
	Following common usage, we will refer to a random map as a JLT when the corresponding distribution satisfies this definition.
\end{definition}
JLTs are usually constructed using simple random matrices, such as Gaussians with i.i.d.\ entries. They have many uses in applications, such as nearest neighbor searching \citep{ailon2009}, least squares regression \citep{avron2010, drineas2011}, sketching of data streams \citep{woodruff2014}, and clustering \citep{makarychev2019}.

When the vectors in the set $\Yc$ in Definition~\ref{def:JLT} have special structure, it is possible to construct a map $f$ that leverages this fact to speed up the computation of $f(\xbf)$ when $\xbf \in \Yc$. One class of vectors with such special structure are the Kronecker vectors $\xbf = \xbf^{(1)} \otimes \xbf^{(2)} \otimes \cdots \otimes \xbf^{(P)}$, where each $\xbf^{(p)} \in \Rb^{I_p}$ and $\otimes$ denotes the Kronecker product. Vectors with Kronecker structure appear in various applications. When matricizing tensors in CP or Tucker format, the resulting matrices have columns which are Kronecker products. Computation with Kronecker vectors therefore feature in algorithms for computing these decompositions \citep{kolda2009} and in related problems like tensor interpolative decomposition \citep{biagioni2015}. They also arise in areas such as higher dimensional numerical analysis \citep{beylkin2002, beylkin2006}, tensor regression \citep{diao2018}, and polynomial kernel approximation in machine learning \citep{pham2013}.
The Kronecker fast Johnson--Lindenstrauss transform (KFJLT) is a map that can be applied very efficiently to Kronecker structured vectors. It was first proposed by \citet{battaglino2018} for solving the least squares problems that arise when computing the CP decomposition of tensors. \citet{battaglino2018} conjectured that the KFJLT is a JLT, but did not provide a proof. 
Recently, \citet{jin2019} provided a proof that the KFJLT indeed is a JLT. 

In this paper, we provide an alternative proof of this fact for when the KFJLT is applied to Kronecker vectors, which is based on a coherence and sampling argument. 
As a stepping stone to proving our result, we also show that the KFJLT is an oblivious subspace embedding for matrices whose columns have Kronecker structure. 
Some ideas that we use in our proof were mentioned in \citep{battaglino2018}. 
Our guarantees are slightly different than those given in \citep{jin2019}: 
Ours have a worse dependence on the ambient dimensions $I_1, I_2, \ldots, I_P$ of the input vectors, but have a better dependence on the accuracy parameter $\varepsilon$. 
The two bounds can be combined into one which yields a better bound overall. 
Another distinction between \citep{jin2019} and our paper is that the result in \citep{jin2019} shows that the KFJLT is a JLT on vectors with arbitrary structure, whereas our result is restricted to vectors with Kronecker structure. 
This means that the guarantees in \citep{jin2019} will be applicable in situations when ours are not. 
For example, KFJLT could be used instead of a standard fast JLT for sketching arbitrary vectors in order to reduce the number of random bits required to construct the sketch. 
However, in certain applications involving arbitrary vectors our guarantees on Kronecker vectors are sufficient. 
For example, when applying a KFJLT sketch to the least squares problem $\min_\xbf \|\Abf \xbf - \ybf\|_2$, where $\Abf$ is a Khatri--Rao product and $\ybf$ is arbitrary, it turns out that our subspace embedding result combined with sampled approximate matrix multiplication ideas from \citep{drineas2006b} is sufficient for deriving guarantees; see Remark~\ref{rem:ls-guarantees} for further details.

%% file: 2-KFJLT-guarantees-related-work-arxiv-v2.tex
As mentioned in the introduction, a JLT can be constructed in many different ways. A popular choice is $f(\xbf) \defeq \Omegabf \xbf /\sqrt{J}$, where $\Omegabf \in \Rb^{J \times I}$ has i.i.d.\ standard normal entries. More generally, the rows of $\Omegabf$ can be chosen to be independent, mean zero, isotropic and sub-Gaussian random vectors in $\Rb^I$ \citep{vershynin2018}. \citet{ailon2009} proposed a fast JLT which leverages the Hadamard transform to achieve a transform that can be applied faster than a general dense matrix $\Omegabf$. 

A concept related to the JLT is subspace embedding. 
\begin{definition}[Subspace embedding \citep{woodruff2014}] \label{def:subspace-embedding}
	A $(1 \pm \varepsilon)$ $\ell_2$-\emph{subspace embedding} for the column space of a matrix $\Xbf \in \Rb^{I \times R}$ is a matrix $\Mbf \in \Rb^{J \times I}$ such that
	\begin{equation} \label{eq:subspace-embedding}
		(\forall \zbf \in \Rb^{R}) \;\;\;\; (1-\varepsilon) \|\Xbf \zbf\|_2^2 \leq \|\Mbf \Xbf \zbf\|_2^2 \leq (1+\varepsilon) \|\Xbf \zbf\|_2^2.
	\end{equation}
	We call a probability distribution on a family $\Fc$ of $J \times I$ matrices an \emph{$(\epsilon,\delta)$ oblivious $\ell_2$-subspace embedding} for $I \times R$ matrices with columns in $\Yc \subset \Rb^{I}$ if, for any matrix 
	\begin{equation}
		\Xbf = [\xbf_1, \; \xbf_2, \ldots, \; \xbf_R] \;\;\;\; \text{with} \;\;\;\; \xbf_1, \xbf_2, \ldots, \xbf_R \in \Yc,
	\end{equation}
	the probability of drawing a matrix $\Mbf \in \Fc$ satisfying \eqref{eq:subspace-embedding} is at least $1-\delta$.\footnote{In the definition of oblivious subspace embedding in Definition~2.2 of \citep{woodruff2014}, the matrix $\Xbf$ can have any structure, which corresponds to $\Yc = \Rb^{I}$. We find it convenient for our purposes to consider random mappings that are oblivious subspace embeddings for matrices with certain structure. 
	}
	Following common usage, we will refer to a random matrix as an oblivious subspace embedding when the corresponding distribution satisfies this definition. Unless specified otherwise, it is assumed that $\Yc = \Rb^I$.
\end{definition}
Thus, a subspace embedding distorts the squared length of a vector in the range of $\Xbf$ by only a small amount. Methods for subspace embedding include leverage score sampling \citep{magdon-ismail2010} and CountSketch \citep{clarkson2017}.
Leverage score sampling is not an oblivious subspace embedding, since the sampling probabilities depend on $\Xbf$. CountSketch, on the other hand, is an oblivious subspace embedding.
Note that a subspace embedding is not necessarily a JLT. CountSketch, for example, is not a JLT \citep{woodruff2014}. For a more complete survey of work related to the JLT and subspace embedding, we refer the reader to the surveys in \citep{mahoney2011, woodruff2014}.

For vectors with Kronecker structure, \citet{sun2018} propose the so called tensor random projection (TRP), whose transpose is a Khatri--Rao product of arbitrary random projection maps. They prove that TRP is a JLT in the special case when the TRP is constructed from two smaller random projections which have entries that are i.i.d.\ sub-Gaussians with zero mean and unit variance. The TRP idea is used in the earlier work \citep{biagioni2015} for tensor interpolative decomposition, but no guarantees are provided there. \citet{rakhshan2020} extend the TRP to allow for a wider range of structured sketches which incorporate CP tensor and tensor-train structure. They assume that the factor matrices and factor tensors for the CP tensor and tensor-train structures, respectively, follow a Gaussian distribution, and prove that their proposed sketches are JLTs. Notably, their results hold for arbitrary orders of the underlying CP tensors and tensor-trains.

\citet{cheng2016} propose an estimated leverage score sampling algorithm for $\ell_2$-regression when the design matrix is a Khatri--Rao product. They use this to speed up the alternating least squares algorithm for computing the tensor CP decomposition. A similar idea is proposed by \citet{diao2019} for $\ell_2$-regression when the design matrix is a Kronecker product.

The papers \citep{pagh2013, pham2013, avron2014, diao2018} develop a method called TensorSketch, which is a variant of CountSketch that can be applied particularly efficiently to matrices whose columns have Kronecker structure. \citet{avron2014} show that TensorSketch is an oblivious subspace embedding, and \citet{diao2018} provide guarantees for $\ell_2$-regression based on TensorSketch. However, just like CountSketch, TensorSketch is not a JLT.

A paper by \citet{iwen2019}, which appeared during the preparation of this paper, considers structured linear embedding operators for tensors. These operators first apply a sketch matrix to each mode of the tensor, then vectorize the result and apply another random sketch. Under certain assumptions on the tensor coherence and sketch matrix properties, they show that their proposed embedding operator is a form of tensor subspace embedding. Combining their approach with results from \citet{jin2019}, they also consider a variant of the KFJLT with improved embedding properties. We make some comparisons between our results and those in \citep{jin2019} and \citep{iwen2019} in Section~\ref{sec:main-results}.

%% file: 3-KFJLT-guarantees-prerequisites-arxiv-v2.tex
We use bold uppercase letters, e.g.\ $\Abf$, to denote matrices; bold lowercase letters, e.g.\ $\abf$, to denote vectors; and regular lowercase letters, e.g.\ $a$, to denote scalars. Regular uppercase letters, e.g.\ $I, J, K$, are usually used to denote the size of vectors and matrices. This means that $I$ is a number and not the identity matrix. Subscripts are used to denote elements of matrices, and a colon denotes all elements in a row or column. For example, for a matrix $\Abf$, $\Abf_{ij}$ is the element on position $(i,j)$, $\Abf_{i:}$ is the $i$th row, and $\Abf_{:j}$ is the $j$th column. Subscripts may be used to label different vectors. Superscripts in parentheses will be used for labeling both matrices and vectors. For example, $\Abf^{(1)}$ and $\Abf^{(2)}$ are two matrices. The norm $\|\cdot\|_2$ denotes the standard Euclidean norm for vectors, and the spectral norm for matrices. For matrices $\Abf \in \Rb^{I \times J}$ and $\Bbf \in \Rb^{K \times L}$, their Kronecker product is denoted by $\Abf \otimes \Bbf \in \Rb^{IK \times JL}$ and is defined as
\begin{equation}
	\Abf \otimes \Bbf \defeq
	\begin{bmatrix}	
		\Abf_{11} \Bbf & \Abf_{12} \Bbf & \cdots & \Abf_{1J} \Bbf\\
		\Abf_{21} \Bbf & \Abf_{22} \Bbf & \cdots & \Abf_{2J} \Bbf\\
		\vdots & \vdots &  & \vdots \\
		\Abf_{I1} \Bbf & \Abf_{I2} \Bbf & \cdots & \Abf_{IJ} \Bbf\\
	\end{bmatrix}.
\end{equation}
For matrices $\Abf \in \Rb^{I \times K}$ and $\Bbf \in \Rb^{J \times K}$, their Khatri--Rao product is denoted by $\Abf \odot \Bbf \in \Rb^{IJ \times K}$ and is defined as 
\begin{equation}
	\Abf \odot \Bbf \defeq
	\begin{bmatrix}
		\Abf_{:1} \otimes \Bbf_{:1} & \Abf_{:2} \otimes \Bbf_{:2} & \cdots & \Abf_{:K} \otimes \Bbf_{:K}
	\end{bmatrix}.
\end{equation}
For a positive integer $n$, we use the notation $[n] \defeq \{1, 2, \ldots, n\}$. We let $\sigma_i(\Abf)$ denote the $i$th singular value of the matrix $\Abf$.

We now introduce the different tools we use to prove our results. 
\begin{definition}[Randomized Hadamard transform \citep{ailon2009}]
	Let $\Hbf \in \Rb^{I \times I}$ be the normalized Hadamard transform, and let $\Dbf \in \Rb^{I \times I}$ be a diagonal matrix with i.i.d.\ Rademacher random variables (i.e., equal to $+1$ or $-1$ with equal probability) on the diagonal. The $I \times I$ \emph{randomized Hadamard transform} is defined as the random map $\xbf \mapsto \Hbf \Dbf \xbf$.
\end{definition}

\begin{definition}[Leverage score, coherence \citep{woodruff2014}] \label{def:leverage-score-coherence}
	Let $\Abf \in \Rb^{I \times R}$ be a matrix, and let $\col(\Abf)$ be a matrix of size $I \times \rank(\Abf)$ whose columns form an orthonormal basis for $\range(\Abf)$. Then 
	\begin{equation}
	\ell_{i}(\Abf) \defeq \| \col(\Abf)_{i:} \|_2^2, \;\;\;\; i \in [I],
	\end{equation}
	is the $i$th \emph{leverage score} of $\Abf$. The \emph{coherence} of $\Abf$ is defined as
	\begin{equation}
	\mu(\Abf) \defeq \max_{i \in [I]} \ell_{i}(\Abf).
	\end{equation}
	The leverage scores, and consequently the coherence, do not depend on the particular basis chosen for the range of $\Abf$ \citep{woodruff2014}, so these quantities are well-defined. The coherence satisfies $\rank(\Abf)/I \leq \mu(\Abf) \leq 1$.
\end{definition}

\begin{definition}[Leverage score sampling \citep{woodruff2014}] \label{def:lev-score-sampling}
	Let $\Abf \in \Rb^{I \times R}$ and $p_i \defeq \ell_i(\Abf)/\rank(\Abf)$ for all $i \in [I]$. Then $\pbf \defeq [p_1, \, p_2, \ldots, \, p_I]$ is a probability distribution on $[I]$. Let $\qbf \defeq [q_1, \, q_2, \ldots, \, q_I]$ be another probability distribution on $[I]$, and suppose that for some $\beta \in (0,1]$ it satisfies $q_i \geq \beta p_i$ for all $i \in [I]$. Let $\vbf \in [I]^{J}$ be a random vector with independent elements satisfying $\Pb(\vbf_{j} = i)= q_i$ for all $(i,j) \in [I]\times[J]$. Let $\bm{\Omega} \in \Rb^{J \times I}$ and $\Rbf \in \Rb^{J \times J}$ be a random sampling matrix and a diagonal rescaling matrix, respectively, defined as
	\begin{equation}
		\bm{\Omega}_{j:} \defeq \ebf_{\vbf_j}^\top \;\;\;\; \text{and} \;\;\;\; \Rbf_{jj} \defeq \frac{1}{\sqrt{J q_{\vbf_j}}} 
	\end{equation}
	for each $j \in [J]$, where $\ebf_{i}$ is the $i$th column of the $I \times I$ identity matrix. The \emph{leverage score sampling} matrix $\Sbf_\qbf \in \Rb^{J \times I}$ is then defined as $\Sbf_\qbf \defeq \Rbf \Omegabf$, where the subscript indicates that the sampling is done according to the distribution $\qbf$.
\end{definition}

\begin{definition}[Kronecker fast Johnson--Lindenstrauss transform \citep{jin2019}] \label{def:KFJLT}
	For each $p \in [P]$, let $\Hbf^{(p)} \Dbf^{(p)}$ be independent randomized Hadamard transforms\footnote{\citet{jin2019} use the discrete Fourier transform instead of the Hadamard transform in their definition.} of size $I_p \times I_p$.
	The \emph{Kronecker fast Johnson--Lindenstrauss transform} (KFJLT) of a vector $\xbf = \xbf^{(1)} \otimes \xbf^{(2)} \otimes \cdots \otimes \xbf^{(P)}$, with $\xbf^{(p)} \in \Rb^{I_p}$, is defined as
	\begin{equation} \label{eq:KFJLT-def}
		\Sbf_\qbf \Big( \kron_{p=1}^P \Hbf^{(p)} \Dbf^{(p)} \Big) \xbf = \Sbf_\qbf \Big( \kron_{p=1}^{P} \Hbf^{(p)} \Dbf^{(p)} \xbf^{(p)} \Big),
	\end{equation}
	where $\Sbf_\qbf \in \Rb^{J \times \tilde{I}}$ is a sampling matrix as in Definition~\ref{def:lev-score-sampling} with $\qbf$ equal to the uniform distribution on $[\tilde{I}]$, where $\tilde{I} \defeq I_1 I_2 \cdots I_P$. The equality in \eqref{eq:KFJLT-def} follows from a basic property of the Kronecker product; see e.g.\ Lemma~4.2.10 in \citep{horn1994}.
\end{definition}
A benefit of the KFJLT is that the Kronecker structured vector does not have to be explicitly computed---it is sufficient to store the smaller vectors $\xbf^{(1)}, \xbf^{(2)}, \ldots, \xbf^{(P)}$. Another benefit is that each randomized Hadamard transform $\Hbf^{(p)} \Dbf^{(p)} \in \Rb^{I_p \times I_p}$ only costs $O(I_p \log I_p)$ to apply to $\xbf^{(p)}$. 

Lemma~\ref{lem:mixing} below is a variant of Lemma~3 in \cite{drineas2011} but with an arbitrary probability of success. The proof is identical to that for Lemma~3 in \cite{drineas2011}---which in turn follows similar reasoning as in the proof of Lemma~1 in \cite{ailon2009}---but using an arbitrary failure probability $\eta$ instead of $1/20$, combined with the definition of leverage score and coherence.
\begin{lemma} \label{lem:mixing}
	Let $\Abf \in \Rb^{I \times R}$ be a matrix and let $\Hbf \Dbf$ be the $I \times I$ randomized Hadamard transform. Then, with probability at least $1-\eta$, the following holds:
	\begin{equation}
	\mu( \Hbf \Dbf \Abf ) \leq \frac{2 R \ln(2IR/\eta)}{I}.
	\end{equation}
\end{lemma}

Lemma~\ref{lem:Cheng} below is a restated version of Theorem~3.3 in \cite{cheng2016}. A similar statement is also made in Lemma~4 in \cite{battaglino2018}.
\begin{lemma} \label{lem:Cheng}
	For each $p \in [P]$, let $\Abf^{(p)} \in \Rb^{I_p \times R}$. Then
	\begin{equation}
	\mu \Big( \kr_{p=1}^P \Abf^{(p)} \Big) \leq \prod_{p=1}^P \mu(\Abf^{(p)}).
	\end{equation}
\end{lemma}

Lemma~\ref{lem:subsampling} below is a slight restatement of Theorem~2.11 in \citep{woodruff2014}, with a careful choice of the constant parameter\footnote{The statement in \citep{woodruff2014} has a constant 144 instead of $8/3$. However, we found that $8/3$ is sufficient under the assumption that $\varepsilon \in (0,1)$. The proof given in \citep{woodruff2014} otherwise remains the same.}. 
\begin{lemma} \label{lem:subsampling}
	Let $\Abf \in \Rb^{I \times R}$ and assume $\varepsilon \in (0,1)$. Suppose 
	\begin{equation} \label{eq:J-bound-subsampling}
		J > \frac{8}{3}\frac{R \ln(2 R/\eta)}{\beta \varepsilon^{2}}
	\end{equation}
	and that $\Sbf_\qbf \in \Rb^{J \times I}$ is a leverage score sampling matrix as in Definition~\ref{def:lev-score-sampling}, where the $\beta$ in that definition is the same as the $\beta$ in \eqref{eq:J-bound-subsampling}. Then, with probability at least $1-\eta$, the following holds:
	\begin{equation}
		(\forall i \in [\rank(\Abf)]) \;\;\;\; 1-\varepsilon \leq \sigma_i^2(\Sbf_\qbf \col(\Abf)) \leq 1+\varepsilon.
	\end{equation}
\end{lemma}

%% file: 4-KFJLT-guarantees-main-result-arxiv-v2.tex
We consider the set
\begin{equation}
	\Yc \defeq \{\xbf \in \Rb^{\tilde{I}} \; : \; \xbf = \xbf^{(1)} \otimes \xbf^{(2)} \otimes \cdots \otimes \xbf^{(P)}, \; \text{with } \xbf^{(p)} \in \Rb^{I_p} \; \text{for each } p \in [P]\}
\end{equation}
of Kronecker vectors. Theorem~\ref{thm:subspace-embedding} shows that the KFJLT is an $(\varepsilon, \delta)$ oblivious $\ell_2$-subspace embedding for matrices whose columns have Kronecker product structure when the embedding dimension $J$ is sufficiently large.
\begin{theorem} \label{thm:subspace-embedding}
	Let $\Xbf = [\xbf_1, \; \xbf_2, \ldots, \; \xbf_R] \in \Rb^{\tilde{I} \times R}$ be a matrix with each column $\xbf_r = \kron_{p=1}^P \xbf_r^{(p)} \in \Yc$.
	For each $p \in [P]$, let $\Hbf^{(p)} \Dbf^{(p)}$ be independent randomized Hadamard transforms of size $I_p \times I_p$, and define 
	\begin{equation}
	\Phibf \defeq \kron_{p=1}^P \Hbf^{(p)} \Dbf^{(p)}.
	\end{equation}
	Let $\Sbf_\qbf \in \Rb^{J \times \tilde{I}}$ be a sampling matrix as in Definition~\ref{def:lev-score-sampling} with $\qbf$ equal to the uniform distribution, and assume $\varepsilon \in (0,1)$.
	If
	\begin{equation} \label{eq:J-bound-subspace-embedding}
		J > \frac{8}{3} \cdot 2^{P} R^{P+1} \varepsilon^{-2} \ln \Big( \frac{2 R (P+1)}{\delta} \Big) \prod_{p=1}^P \ln\Big(\frac{2 I_p R (P+1)}{\delta}\Big),
	\end{equation}
	then the following holds with probability at least $1-\delta$:
	\begin{equation}
		(\forall \zbf \in \Rb^R) \;\;\;\; (1-\varepsilon) \|\Xbf \zbf\|_2^2 \leq \|\Sbf_\qbf \Phibf \Xbf \zbf\|_2^2 \leq (1+\varepsilon) \|\Xbf \zbf\|_2^2.
	\end{equation}
\end{theorem}
\begin{proof}
	If all columns of $\Xbf$ are the zero vector, the claim is trivially true. So we now assume that at least one column of $\Xbf$ is nonzero.
	Note that $\Xbf = \kr_{p=1}^P \Xbf^{(p)}$, where each $\Xbf^{(p)} \defeq [\xbf_1^{(p)}, \; \xbf_2^{(p)}, \ldots, \; \xbf_R^{(p)}]$. By Lemma~\ref{lem:mixing}, for a fixed $p \in [P]$, the following holds with probability at least $1-\eta$:
	\begin{equation}
		\mu(\Hbf^{(p)} \Dbf^{(p)} \Xbf^{(p)}) \leq \frac{2 R \ln(2 I_p R/\eta)}{I_p}.
	\end{equation}
	Hence, taking a union bound, the following holds with probability at least $1 - P\eta$:
	\begin{equation}
		(\forall p \in [P]) \;\;\;\; \mu(\Hbf^{(p)} \Dbf^{(p)} \Xbf^{(p)}) \leq \frac{2 R \ln(2 I_p R/\eta)}{I_p}.
	\end{equation}
	Now applying Lemma~\ref{lem:Cheng}, we have that the following holds with probability at least $1-P\eta$:
	\begin{equation} \label{eq:coherence-bound}
	\mu(\Phibf \Xbf) = \mu\Big(\kr_{p=1}^P \Hbf^{(p)} \Dbf^{(p)} \Xbf^{(p)}\Big) \leq \prod_{p=1}^P \mu(\Hbf^{(p)} \Dbf^{(p)} \Xbf^{(p)}) \leq \frac{1}{\tilde{I}} \prod_{p=1}^P 2R \ln(2I_p R/\eta).
	\end{equation}
	For $i \in [\tilde{I}]$, let $p_i = \ell_i(\Phibf \Xbf) / \rank(\Phibf \Xbf)$. 
	Since $\Phibf$ is a Kronecker product of orthogonal matrices, $\Phibf$ is also orthogonal \citep{loan2000}, and since $\Xbf$ is nonzero, it follows that $\rank(\Phibf \Xbf) \geq 1$, so $p_i$ is well defined.
	Instead of sampling according to the unknown distribution $[p_1, \; p_2, \ldots, \; p_{\tilde{I}}]$, we sample according to the uniform distribution $\qbf = [q_1, \; q_2, \ldots, \; q_{\tilde{I}}]$. To get guarantees, we want to apply Lemma~\ref{lem:subsampling}.
	To do this, we first need to find some $\beta \in (0, 1]$ such that
	\begin{equation}
		(\forall i \in [\tilde{I}]) \;\;\;\; q_i = \frac{1}{\tilde{I}} \geq \beta p_i.
	\end{equation}
	From \eqref{eq:coherence-bound}, the following holds with probability at least $1 - P \eta$:
	\begin{equation}
		p_i = \frac{\ell_i(\Phibf \Xbf)}{\rank(\Phibf \Xbf)} \leq \mu(\Phibf \Xbf) \leq \frac{1}{\tilde{I}} \prod_{p=1}^P 2R \ln(2I_p R/\eta).
	\end{equation}
	Hence, choosing $\beta$ such that
	\begin{equation}
	\beta^{-1} = \prod_{p=1}^P 2 R \ln (2 I_p R / \eta)
	\end{equation}
	ensures that $q_i = 1/\tilde{I} \geq \beta p_i$ for all $i \in [\tilde{I}]$ with probability at least $1-P \eta$. Let $\alpha \defeq \rank(\Phibf \Xbf) = \rank(\Xbf)$, and let $\Ubf \Sigmabf \Vbf^\top = \Xbf$ be the SVD of $\Xbf$ with $\Ubf \in \Rb^{\tilde{I} \times \alpha}$, $\Sigmabf \in \Rb^{\alpha \times \alpha}$ and $\Vbf \in \Rb^{R \times \alpha}$. Note that the columns of $\Phibf \Ubf$ form an orthonormal basis for $\range(\Phibf \Xbf)$. Hence, we can choose $\col(\Phibf \Xbf) = \Phibf \Ubf$. 
	Using Lemma~\ref{lem:subsampling}, with $\Abf = \Phibf \Xbf \in \Rb^{\tilde{I} \times R}$, it follows that if
	\begin{equation} \label{eq:J-bound-subspace-embedding-0}
	J > \frac{8}{3} \cdot 2^P R^{P+1} \varepsilon^{-2} \ln(2 R / \eta) \prod_{p=1}^P \ln(2 I_p R / \eta),
	\end{equation}
	then the following holds with probability at least $1-(P+1)\eta$:
	\begin{equation}
	(\forall i \in [\alpha]) \;\;\;\;\ 1-\varepsilon \leq \sigma_i^2(\Sbf_\qbf \Phibf \Ubf) \leq 1+\varepsilon.
	\end{equation}
	By the minimax characterization of singular values (see e.g.\ Theorem~8.6.1 in \citep{golub2013}), it follows that
	\begin{equation}
	(\forall \wbf \in \Rb^\alpha) \;\;\;\; (1-\varepsilon) \|\wbf\|_2^2 \leq \|\Sbf_\qbf \Phibf \Ubf \wbf \|_2^2 \leq (1+\varepsilon) \|\wbf\|_2^2
	\end{equation}
	holds with probability at least $1-(P+1)\eta$.
	In particular, for any $\zbf \in \Rb^{R}$, this is true for $\wbf = \Sigmabf \Vbf^\top \zbf \in \Rb^\alpha$. Consequently, 
	\begin{equation}
	(\forall \zbf \in \Rb^R) \;\;\;\; (1-\varepsilon) \|\Sigmabf \Vbf^\top \zbf\|_2^2 \leq \|\Sbf_\qbf \Phibf \Ubf \Sigmabf \Vbf^\top \zbf \|_2^2 \leq (1+\varepsilon) \|\Sigmabf \Vbf^\top \zbf\|_2^2,
	\end{equation} 
	or equivalently, 
	\begin{equation}
	(\forall \zbf \in \Rb^R) \;\;\;\; (1-\varepsilon) \|\Xbf \zbf\|_2^2 \leq \|\Sbf_\qbf \Phibf \Xbf \zbf \|_2^2 \leq (1+\varepsilon) \|\Xbf \zbf\|_2^2,
	\end{equation} 
	holds with probability at least $1-(P+1)\eta = 1-\delta$, where $\delta \defeq (P+1)\eta$. Replacing $\eta = \delta/(P+1)$ in \eqref{eq:J-bound-subspace-embedding-0} gives \eqref{eq:J-bound-subspace-embedding}.	
\end{proof}

The following theorem is our main result. It shows that the KFJLT is a $\JLT(\varepsilon, \delta, N)$ on $\Yc$ when the embedding dimension $J$ is sufficiently large. 
\begin{theorem} \label{thm:KFJLT}
	Let $\Xc \subseteq \Yc$ consist of $N$ distinct vectors with Kronecker structure. 
	Let $\Phibf$ be defined as in Theorem~\ref{thm:subspace-embedding}. 
	Let $\Sbf_\qbf \in \Rb^{J \times \tilde{I}}$ be a sampling matrix as in Definition~\ref{def:lev-score-sampling} with $\qbf$ equal to the uniform distribution, and assume $\varepsilon \in (0,1)$. If
	\begin{equation} \label{eq:J-bound}
	J > \frac{16}{3} \cdot 4^P \varepsilon^{-2} \ln \Big( \frac{4 N^2 (P+1)}{\delta} \Big) \prod_{p=1}^P \ln \Big(\frac{4 I_p N^2 (P+1)}{\delta} \Big),
	\end{equation} 
	then the following holds with probability at least $1-\delta$:
	\begin{equation} \label{eq:KFJLT-bounds}
	(\forall \xbf, \ybf \in \Xc) \;\;\;\; (1-\varepsilon) \| \xbf - \ybf \|_2^2 \leq \|\Sbf_\qbf \Phibf \xbf - \Sbf_\qbf \Phibf \ybf \|_2^2 \leq (1+\varepsilon) \| \xbf - \ybf \|_2^2.
	\end{equation}
\end{theorem}

\begin{proof}
	Fix $\xbf, \ybf \in \Xc$ and set $\Xbf \defeq [\xbf, \; \ybf]$. From Theorem~\ref{thm:subspace-embedding}, we know that if
	\begin{equation} \label{eq:J-bound-2}
		J > \frac{8}{3} \cdot 2^P 2^{P+1} \varepsilon^{-2} \ln \Big( \frac{4 (P+1)}{\eta} \Big) \prod_{p=1}^P \ln \Big( \frac{4 I_p (P+1)}{\eta} \Big),
	\end{equation} 
	then the following holds with probability at least $1 - \eta$:
	\begin{equation}
		(\forall \zbf \in \Rb^2) \;\;\;\; (1-\varepsilon) \|\Xbf \zbf\|_2^2 \leq \| \Sbf_\qbf \Phibf \Xbf \zbf \|_2^2 \leq (1+\varepsilon) \|\Xbf \zbf\|_2^2.
	\end{equation}
	In particular, setting $\zbf = [1, \; -1]^\top$, we have that, with probability at least $1-\eta$,
	\begin{equation}
		(1-\varepsilon) \| \xbf - \ybf \|_2^2 \leq \| \Sbf_\qbf \Phibf \xbf - \Sbf_\qbf \Phibf \ybf \|_2^2 \leq (1+\varepsilon) \|\xbf - \ybf\|_2^2.
	\end{equation}
	Taking a union bound over all distinct $N^2 - N$ pairs $\xbf, \ybf \in \Xc$, we have that \eqref{eq:KFJLT-bounds} holds with probability at least $1 - N^2 \eta = 1 - \delta$, where $\delta \defeq N^2 \eta$. Replacing $\eta = \delta / N^2$ in \eqref{eq:J-bound-2} gives \eqref{eq:J-bound}.
\end{proof}

Assuming $N > \max\{P, 4\}$, the bound in \eqref{eq:J-bound} can be simplified to
\begin{equation} \label{eq:J-bound-simplified}
	J > C_1 \varepsilon^{-2} C_2^{P} \log\Big(\frac{N}{\delta}\Big) \prod_{p=1}^P \log\Big( \frac{I_p N}{\delta} \Big),
\end{equation}
where $C_1$ and $C_2$ are absolute constants.
For comparison, and expressed in the same notation as in this paper, the bound on $J$ in Theorem~2.1 in \cite{jin2019} needed to guarantee that the KFJLT is a $\JLT(\varepsilon, \delta + 2^{-\Omega(\log \tilde{I})}, N)$ on $\Rb^{\tilde{I}}$ is of the form
\begin{equation} \label{eq:J-bound-jin}
J > C \varepsilon^{-2} \log^{2P-1}\Big( \frac{P N}{\delta} \Big) \log^4\bigg( \varepsilon^{-1} \log^P\Big(\frac{P N}{\delta}\Big) \bigg) \log \Big( \prod_{p=1}^P I_p \Big),
\end{equation}
where $C$ is an absolute constant. Whether \eqref{eq:J-bound-simplified} or \eqref{eq:J-bound-jin} yield a better bound depends on the various parameters.
For example, the bound in \eqref{eq:J-bound-jin} has a nicer dependence on the dimension sizes $I_1, I_2, \ldots,I_P$ than the bound in \eqref{eq:J-bound-simplified}. Indeed, the term 
\begin{equation}
	\log \Big( \prod_{p=1}^P I_p \Big) = \sum_{p=1}^P \log(I_p)
\end{equation}
in \eqref{eq:J-bound-jin} is a sum of logs of $I_1, I_2, \ldots,I_P$, whereas the term
\begin{equation}
	\prod_{p=1}^P \log\Big( \frac{I_p N}{\delta} \Big)
\end{equation}
in \eqref{eq:J-bound-simplified} is a product of logs of $I_1 N/\delta, I_2 N/\delta, \ldots, I_P N/\delta$. 
On the other hand, the bound in \eqref{eq:J-bound-simplified} has a nicer dependence on $\varepsilon$ than \eqref{eq:J-bound-jin} does. 
Indeed, \eqref{eq:J-bound-simplified} contains the term $(\varepsilon^{-2})$ whereas \eqref{eq:J-bound-jin} contains the term $(\varepsilon^{-2} \log^4(\varepsilon^{-1}))$. 
These two bounds can therefore be combined to yield a better bound on the size of $J$ required to ensure that the KFJLT is a JLT on $\Yc$. 

As noted by \citet{iwen2019}, one of the intermediate embedding dimension results in Theorem~1 of their paper can be translated to a subspace embedding result of the same flavor as what we present in Theorem~\ref{thm:subspace-embedding}. Their result has a better dependence on $R$: It is proportional to $R^2$ while our bound in \eqref{eq:J-bound-subspace-embedding} is proportional to $R^{P+1}$. Moreover, their results hold for a large family of sketch matrices, whereas our result is limited to the KFJLT sketch. On the other hand, their bound has worse dependence on $\varepsilon$: It is proportional to $\varepsilon^{-2P}$ while our bound is proportional to $\varepsilon^{-2}$. Their guarantees also require a coherence assumption. In Theorems~2 and 8 of their paper, they provide guarantees for a variant of KFJLT, with an intermediate embedding result that corresponds to a subspace embedding variant of Theorem~2.1 in \citep{jin2019}.

\begin{remark} \label{rem:ls-guarantees}
	Our coherence and sampling argument can also be used to provide guarantees for sketched least squares regression. 
	Let $\Xbf$, $\Phibf$ and $\Sbf_\qbf$ be defined as in Theorem~\ref{thm:subspace-embedding} and suppose $\Xbf$ is full rank, let $\ybf \in \Rb^{\tilde{I}}$ be an arbitrary vector with no assumptions on its structure, and let $\varepsilon \in (0,1)$.
	One can show that if $J$ is large enough, then $\|\Xbf \hat{\zbf} - \ybf\|_2 \leq (1+\varepsilon) \OPT$, where $\hat{\zbf} \defeq \argmin_\zbf\|\Sbf_\qbf \Phibf (\Xbf \zbf - \ybf)\|_2$ and $\OPT \defeq \min_\zbf \|\Xbf \zbf - \ybf\|_2$. This can be done by following the same arguments as in the proof of Theorem~2 in Section~4 of \citep{drineas2011}. Let $\Ubf$ be the top $R$ singular values of $\Xbf$ and let $\ybf_\perp$ be the portion of $\ybf$ which is perpendicular to $\range(\Ubf)$. The proof boils down to showing that $\sigma^2_R(\Sbf_\qbf \Phibf \Ubf) \geq 1/\sqrt{2}$ and $\|\Ubf^\top (\Sbf_\qbf \Phibf)^\top (\Sbf_\qbf \Phibf) \ybf_\perp\|_2^2 \leq \varepsilon \cdot \OPT^2/2$ with high probability for sufficiently large $J$. The first statement follows directly from Theorem~\ref{thm:subspace-embedding}, and the second statement follows from Monte Carlo sampling results in \citep{drineas2006b} which require no information about $\ybf_\perp$. We refer to \citep{drineas2011} for further details.
\end{remark}

%% file: 5-KFJLT-guarantees-experiments-arxiv-v2.tex
We present results from experiments on both synthetic and real-world data. These experiments were implemented in Matlab.\footnote{Our code is available online at \url{https://github.com/OsmanMalik/kronecker-sketching}.}

\subsection{Experiment 1: Synthetic Data}

In this section, we present the results from an experiment which compares five different sketches when applied to random Kronecker vectors with three different random distributions. The five methods we compare are the following.
\begin{itemize}
	\item \textbf{Gaussian} sketch uses an unstructured $J \times \tilde{I}$ matrix with i.i.d.\ standard normal entries which are scaled by $1/\sqrt{J}$. This approach is not scalable, but interesting to use as a baseline in this experiment.
	\item \textbf{KFJLT} is the sketch discussed in this paper, and which is defined in Definition~\ref{def:KFJLT}, with the only difference that the uniform sampling of rows is done without replacement.
	\item \textbf{TRP} is the method proposed in \citep{sun2018}. As sub-matrices, we use matrices with i.i.d.\ standard normal entries of size $I_p \times J$ and rescale appropriately.
	\item \textbf{TensorSketch} is the method developed in \citep{pagh2013, pham2013, avron2014, diao2018}.
	\item \textbf{Sampling} is the method proposed by \citep{cheng2016}. It computes an estimate of the leverage scores for each row of the matrix to be sampled and uses these to compute a distribution which is used for sampling.
\end{itemize}
All of these methods, except the first, are specifically designed for sketching vectors with Kronecker structure. As input, we use random Kronecker vectors $\xbf = \kron_{p=1}^3 \xbf^{(p)}$, where each $\xbf^{(p)} \in \Rb^{16}$ has one of the following distributions.
\begin{itemize}
	\item Each $\xbf^{(p)}$ has i.i.d.\ standard normal entries. 
	\item Each $\xbf^{(p)}$ is sparse with three nonzero elements, which are independent and normally distributed with mean zero and standard deviation 100. The positions of the three nonzero elements are drawn uniformly at random without replacement.
	\item Each $\xbf^{(p)}$ contains a single nonzero entry, which is chosen uniformly at random. This nonzero entry is equal to 100.
\end{itemize}
Sparse Kronecker vectors are interesting in many data science applications, and arise in decomposition of sparse tensors, for example.

In the experiment, we draw two random vectors $\xbf, \ybf \in \Rb^{4096}$, each of which is a Kronecker product of three smaller vectors of length $16$, drawn according to one of the three distributions above. For each of the five sketches $f$ and for some embedding dimension $J$, we then compute how well they preserve the distance between $\xbf$ and $\ybf$ by computing the quantity
\begin{equation} \label{eq:distortion-metric}
	\Big| \frac{\|f(\xbf)-f(\ybf)\|_2}{\|\xbf - \ybf\|_2} - 1 \Big|.
\end{equation}
For each of the three distributions and each embedding dimension
\begin{equation} \label{eq:J-range}
	J \in \{100, 200, \ldots, 1000\}
\end{equation}
we repeat this 1000 times and compute the mean, standard deviation and maximum of \eqref{eq:distortion-metric} over those 1000 trials. 
For $J$ as in \eqref{eq:J-range}, applying one of the sketches to $\xbf$ and $\ybf$ reduces the number of entries in those vectors by between 76\% (for $J=1000$) and 98\% (for $J=100$).
Figures~\ref{fig:experiment1-normal}--\ref{fig:experiment1-large-single} present the results for each of the three distributions.

\begin{figure}[ht!]
	\centering  
	\includegraphics[width=1\textwidth]{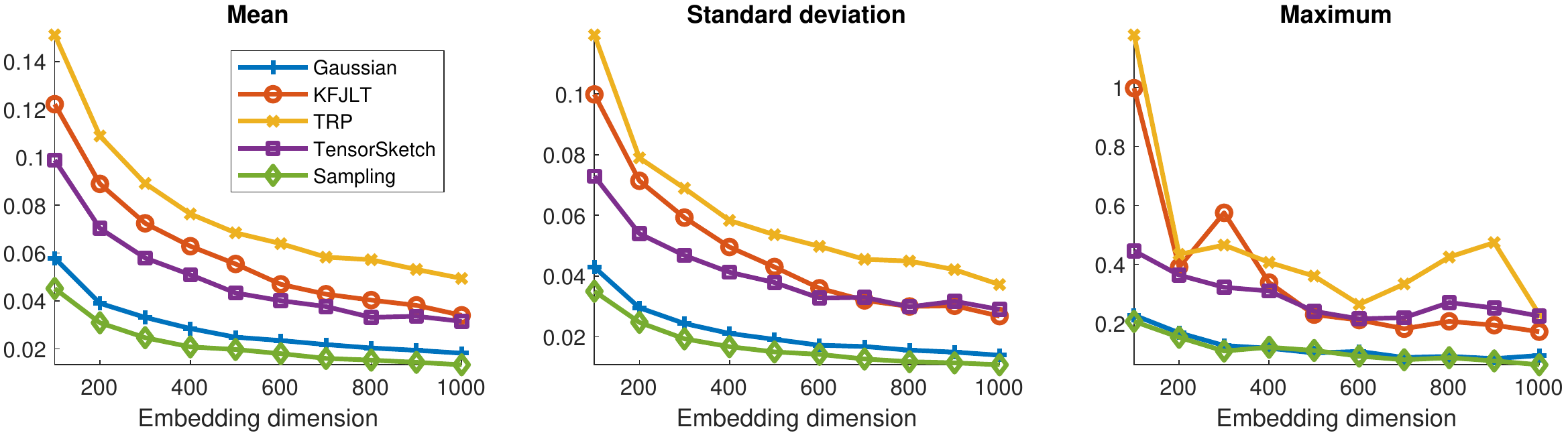}
	\caption{Mean, standard deviation, and maximum of the quantity in \eqref{eq:distortion-metric} over 1000 trials when the test vectors are Kronecker products of vectors with i.i.d.\ standard normal entries.}
	\label{fig:experiment1-normal}
\end{figure}

\begin{figure}[ht!]
	\centering  
	\includegraphics[width=1\textwidth]{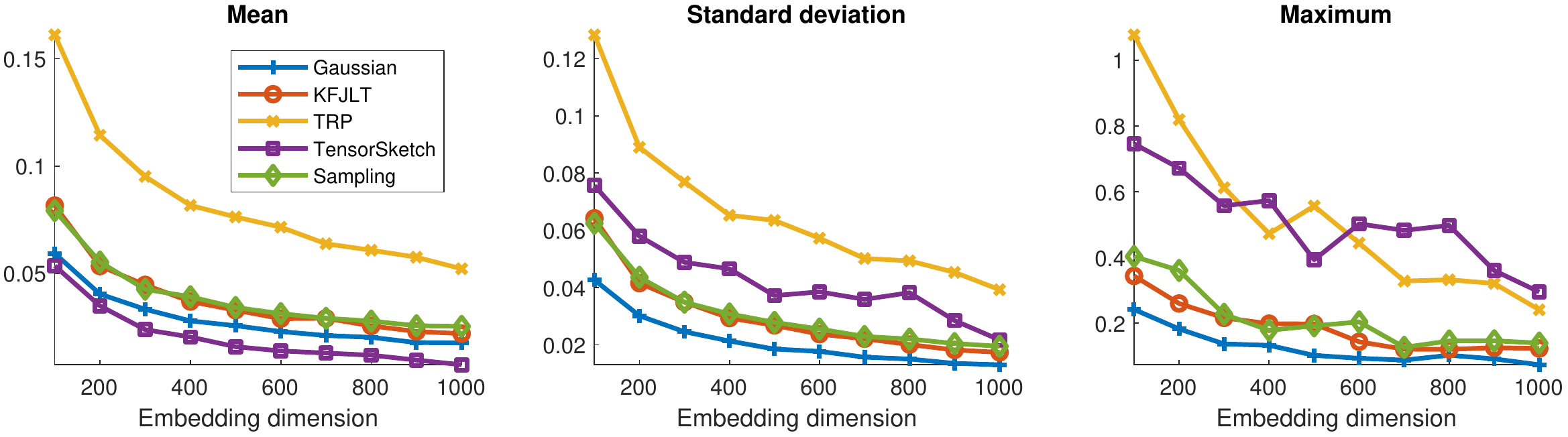}
	\caption{Mean, standard deviation, and maximum of the quantity in \eqref{eq:distortion-metric} over 1000 trials when the test vectors are Kronecker products of vectors with three nonzero elements which are independent and normally distributed with mean zero and standard deviation 100.}
	\label{fig:experiment1-sparse}
\end{figure}

\begin{figure}[ht!]
	\centering  
	\includegraphics[width=1\textwidth]{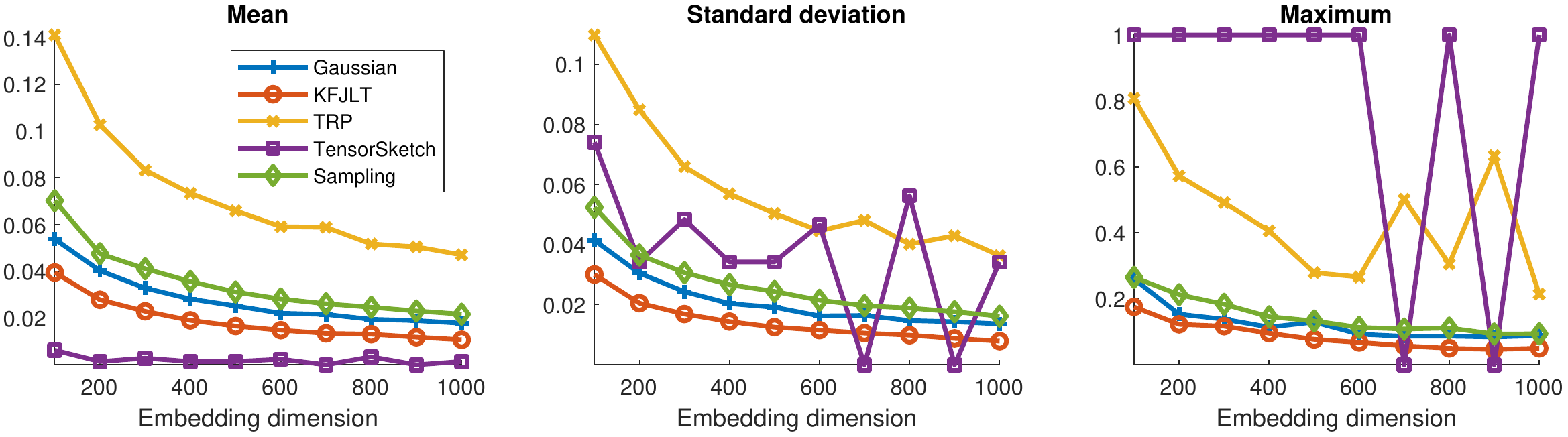}
	\caption{Mean, standard deviation, and maximum of the quantity in \eqref{eq:distortion-metric} over 1000 trials when the test vectors are Kronecker products of vectors containing a single nonzero entry equal to 100.}
	\label{fig:experiment1-large-single}
\end{figure}

No one method produces the best results for all three distributions. The leverage score sampling approach does very well on dense vectors, even outperforming the Gaussian sketch, but does less well on sparser inputs. Although TensorSketch has an impressive mean performance on the two sparser inputs, it sometimes produces high distortion rates on those inputs. On the sparser inputs, the KFJLT seems to strike the best balance between mean and worst case performance. TRP does poorly for all three distribution types.

\subsection{Experiment 2: MNIST Handwritten Digits}

In this experiment, we consider a subset of the MNIST Handwritten Digits dataset \citep{lecun1998}, which is a standard benchmark dataset in machine learning.\footnote{We downloaded the MNIST dataset using the scripts provided at \url{https://github.com/sunsided/mnist-matlab}.} 
The dataset consists of images of handwritten digits between 0 and 9. 
Each image is in gray scale and of size 28 by 28 pixels. 
To make the image width and height powers of two, we pad the images with zeros so that their size is 32 by 32 pixels.
We arrange 100 images depicting fours into a tensor $\Xe \in \Rb^{32 \times 32 \times 100}$ and 100 images depicting nines into another tensor $\Ye \in \Rb^{32 \times 32 \times 100}$. 
Handwritten fours and nines can look quite similar and can be difficult to distinguish, which is why we choose this particular pair of digits.
We then compute a rank-10 approximate CP decomposition of each tensor using \verb|cp_als| in the Tensor Toolbox for Matlab \citep{bader2006,bader2015}. These take the form
\begin{equation}
\begin{aligned}
	\hat{\Xe} &\defeq \sum_{r=1}^{10} \Abf^{(1)}_{:r} \circ \Abf^{(2)}_{:r} \circ \Abf^{(3)}_{:r} \approx \Xe, \\
	\hat{\Ye} &\defeq \sum_{r=1}^{10} \Bbf^{(1)}_{:r} \circ \Bbf^{(2)}_{:r} \circ \Bbf^{(3)}_{:r} \approx \Ye, \\
\end{aligned}
\end{equation}
where $\circ$ denotes outer product, and each $\Abf^{(1)}, \Bbf^{(1)}, \Abf^{(2)}, \Bbf^{(2)} \in \Rb^{32 \times 10}$ and $\Abf^{(3)}, \Bbf^{(3)} \in \Rb^{100 \times 10}$ are called factor matrices; see \citep{kolda2009} for further details on tensor decomposition.
Notice that the factor matrices require much less storage than the original tensors. Figure~\ref{fig:4-9} shows an example of a four and a nine in the MNIST dataset, and their corresponding approximations in the CP tensors.

\begin{figure}[ht!]
	\centering  
	\includegraphics[width=1\textwidth]{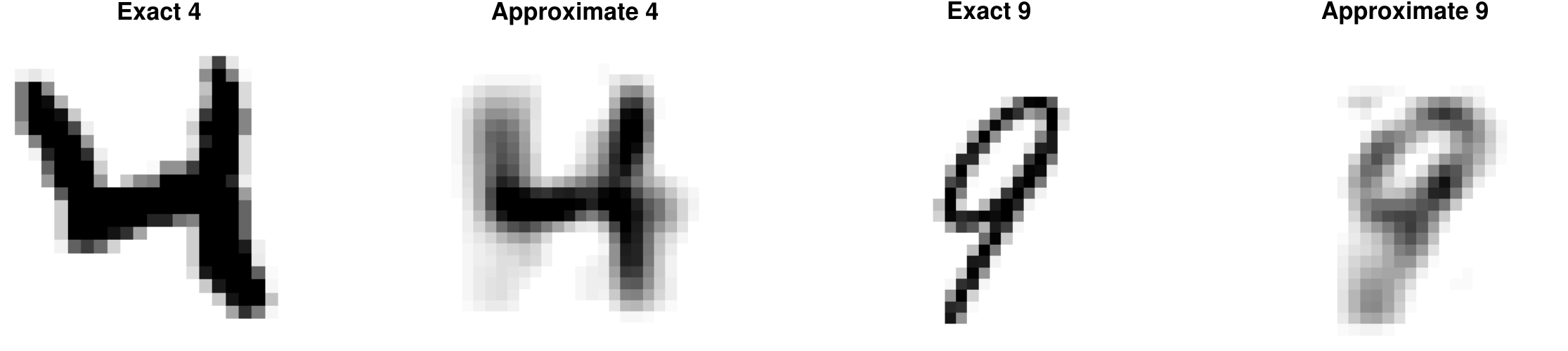}
	\caption{Example of a four and a nine with their corresponding approximations.}
	\label{fig:4-9}
\end{figure}

In this experiment, we apply the KFJLT, TRP, TensorSketch and sampling sketches to $\hat{\Xe}$ and $\hat{\Ye}$ to see how well they preserve the distance between these tensors. 
The four sketches can be applied efficiently to $\hat{\Xe}$ and $\hat{\Ye}$ in their decomposed form.
To see this, define $\Xbf \defeq \Abf^{(1)} \odot \Abf^{(2)} \odot \Abf^{(3)} \in \Rb^{102400 \times 10}$ and $\Ybf \defeq \Bbf^{(1)} \odot \Bbf^{(2)} \odot \Bbf^{(3)} \in \Rb^{102400 \times 10}$, and let $\ubf \in \Rb^{20}$ denote a column vector with elements $\ubf_i = 1$ if $1 \leq i \leq 10$ and $\ubf_i = -1$ if $11 \leq i \leq 20$.
Then the following relation holds:
\begin{equation}
	\|\hat{\Xe} - \hat{\Ye}\|_\F = \| [ \Xbf, \, \Ybf ] \, \ubf \|_2,
\end{equation}
where $\|\cdot\|_\F$ denotes the tensor Frobenius norm. 
Since the columns of $[ \Xbf, \, \Ybf ]$ are Kronecker products, each of the four sketches under consideration can be applied efficiently to this matrix. 
The Gaussian sketch requires too much memory and is therefore not considered.  

For any sketch $\Mbf$ with the property 
\begin{equation} \label{eq:M-tensor-sketch}
	\| [ \Xbf, \, \Ybf ] \, \ubf \|_2 \approx \| \Mbf \, [ \Xbf, \, \Ybf ] \, \ubf \|_2,
\end{equation}
we may use $\| \Mbf \, [ \Xbf, \, \Ybf ] \, \ubf \|_2$ as an estimate for $\|\hat{\Xe} - \hat{\Ye}\|_\F$. 
In the case of KFJLT, a guarantee of the form \eqref{eq:M-tensor-sketch} follows from Theorem~\ref{thm:subspace-embedding} when $J$ is large enough.
For each of the four sketches and some embedding dimension $J$, we compute the quantity
\begin{equation} \label{eq:distortion-metric-2}
\Big| \frac{\| \Mbf \, [ \Xbf, \, \Ybf ] \, \ubf \|_2}{\|\hat{\Xe} - \hat{\Ye}\|_\F} - 1 \Big|
\end{equation}
as a measure of performance. 
For each embedding dimension
\begin{equation} \label{eq:J-range-2}
	J \in \{ 100, 200, \cdots, 5000 \}
\end{equation}
we repeat this 1000 times and compute the mean, standard deviation and maximum of \eqref{eq:distortion-metric-2} over those 1000 trials. 
The pair $(\hat{\Xe}, \hat{\Ye})$ remains the same in all trials.
For $J$ as in \eqref{eq:J-range-2}, applying one of the sketches to $[ \Xbf, \, \Ybf ]$ reduces the number of rows by between 95\% (for $J = 5000$) and 99.9\% (for $J = 100$).
Figure~\ref{fig:experiment2} presents the results of the experiment. 
\begin{figure}[ht!]
	\centering  
	\includegraphics[width=1\textwidth]{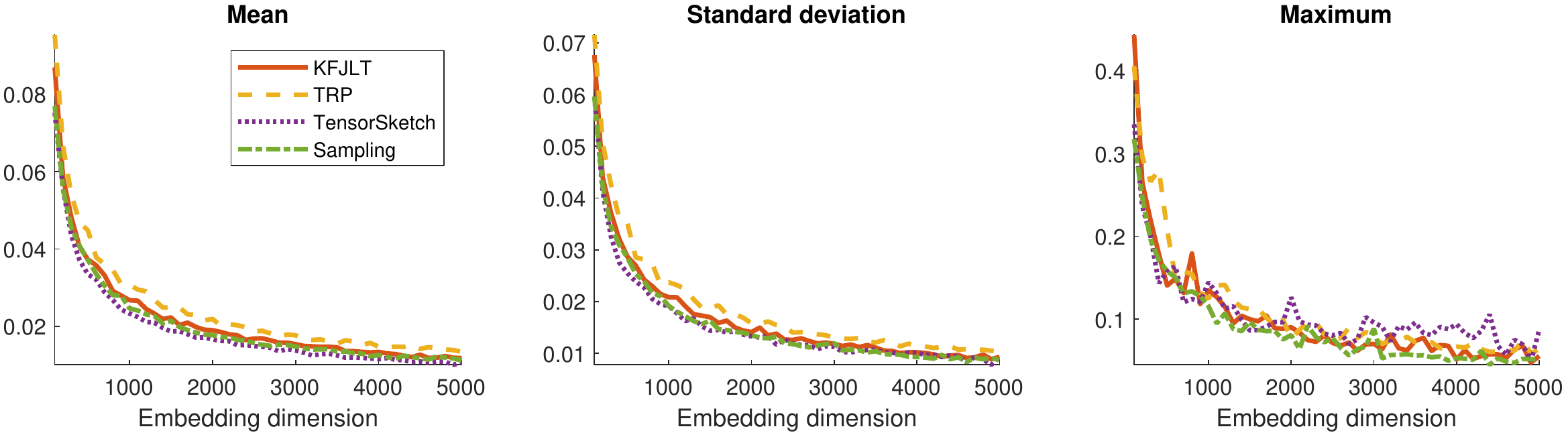}
	\caption{Mean, standard deviation, and maximum of the quantity in \eqref{eq:distortion-metric-2} over 1000 trials.}
	\label{fig:experiment2}
\end{figure}
All methods have similar performance. 
The fact that the factor matrices are mostly dense (some rows are zero due to the padding) may explain why the occasional large errors for TensorSketch observed in Figure~\ref{fig:experiment1-large-single} are avoided. 
Our results here indicate that the different sketches may have more similar performance on vectors like $[ \Xbf, \, \Ybf ] \, \ubf$ which have less structure than the Kronecker vectors in the synthetic experiment.
\citet{jin2019} made the related observation that KFJLT does a better job of embedding unstructured vectors than Kronecker structured ones; see Section~5.2 in their paper for further details.

%% file: 6-KFJLT-guarantees-conclusions-arxiv-v2.tex
We have presented a coherence and sampling argument for showing that the KFJLT is a Johnson--Lindenstrauss transform on vectors with Kronecker structure. 
Since our bound on the embedding dimension is different from the one in the recent paper by \citet{jin2019}, it can be combined with the bound from that paper to yield a better bound overall. 
As a stepping stone to proving our result, we also showed that the KFJLT is a subspace embedding for matrices whose columns are Kronecker products.

We provided results from numerical experiments which compare five different sketches, four of which are designed to be particularly efficient for sketching of Kronecker structured vectors. 
The first experiment was done on Kronecker vectors with three different random distributions. 
The second experiment was done on two CP tensors, each approximating a tensor containing digits from the MNIST dataset. 
In the first experiment, there was a clear difference in performance between different sketches, although no single method outperformed all others for all three vector distributions. 
In the second experiment, all methods performed similarly except the unstructured Gaussian sketch which was not included in the experiment due to its high memory usage. 
We believe that there is a need for a more comprehensive comparison of sketches for structured data to help practitioners choose the best sketch for their particular needs.